\documentclass[12pt]{amsart}
\newtheorem{thm}{Theorem}

\newcommand{\im}{\operatorname{im}}

\begin{document}
\title{Extensions of the coeffective complex}
\author[Michael Eastwood]{Michael Eastwood${}^\dagger$}
\address{\hskip-\parindent
Mathematical Sciences Institute\\
Australian National University,\newline
ACT 0200, Australia}
\email{meastwoo@member.ams.org}
\footnotetext{Support by the Australian Research Council.}
\begin{abstract}
The coeffective differential complex on a symplectic manifold is extended both 
in length and in scope, unifying the constructions of various other authors.
\end{abstract}
\renewcommand{\subjclassname}{\textup{2010} Mathematics Subject Classification}
\maketitle

\section{Introduction}\label{intro}
This article is both an addendum to~\cite{BEGN} and a precursor to~\cite{ES}.
In~\cite{BEGN}, we discussed the construction of differential complexes on
manifolds equipped with various geometric structures. Mostly, these geometries
were {\em parabolic\/}~\cite{CS} but there were two exceptions, specifically
{\em contact\/} geometry for which there is the {\em Rumin\/} complex~\cite{R}
and {\em symplectic\/} for which there is a very similar complex~\cite{Se},
which we dubbed the {\em Rumin-Seshadri\/} complex (it was independently
discovered by Tseng and Yau~\cite{TY}). This article extends the realm of these
complexes, specifically covering {\em conformally symplectic\/} manifolds and
{\em conformally calibrated $G_2$\/} manifolds (see, for example, \cite{Ba,V}
and \cite{FU}, respectively).

In~\cite{Bo}, T.~Bouche introduced a differential complex naturally defined on
any symplectic manifold~$M$ and coined the term {\em coeffective complex\/} for
it (see also~\cite{FIL}). If $M$ has dimension~$2n$, then it is the subcomplex
of the second half of the de~Rham complex
$$\begin{array}{lllllllllllllll}
\Lambda^{n}&\xrightarrow{\,d\,}&\Lambda^{n+1}&\xrightarrow{\,d\,}
&\cdots&\xrightarrow{\,d\,}&\Lambda^{2n-2}&\xrightarrow{\,d\,}
&\Lambda^{2n-1}&\xrightarrow{\,d\,}&\Lambda^{2n}&\to&0\\
\cup&&\cup&&&&\cup&&\,\|&&\,\|\\
\Lambda_\perp^{n}&\to&\Lambda_\perp^{n+1}&\to
&\cdots&\to&\Lambda_\perp^{2n-2}&\to
&\Lambda^{2n-1}&\to&\Lambda^{2n}&\to&0,
\end{array}$$
where, if $J$ denotes the symplectic form, then the bundle $\Lambda_\perp^k$ 
is defined as the kernel of 
$\Lambda^k\xrightarrow{\,J\wedge\underbar\enskip\,}\Lambda^{k+2}$. Under the 
canonical isomorphisms 
$$\underbrace{J\wedge J\wedge\cdots\wedge J}_{n-k}\wedge\underbar\enskip:
\Lambda^k\stackrel{\simeq\quad}{\longrightarrow}\Lambda^{2n-k}\enskip
\mbox{for }k=0,1,2,\ldots n$$
the bundle $\Lambda_\perp^{2n-k}$ may equally well be regarded as a subbundle
of~$\Lambda^k$, which we shall write as $\Lambda_\perp^k$ and, as such,
provides a natural complement to the range of
$\Lambda^{k-2}\xrightarrow{\,J\wedge\underbar\enskip\,}\Lambda^k$ for
$k=2,3,\ldots,n$. Using indices (more precisely, {\em abstract indices\/} in
the sense of~\cite{OT}), sections of the bundle $\Lambda_\perp^k$ for
$k=2,3,\ldots,n$ are precisely the $k$-forms that are {\em trace-free\/} with 
respect to $J_{ab}$, i.e.\
$$J^{ab}\omega_{abc\cdots d}=0,$$
where $J^{ab}$ is the inverse of $J_{ab}$ (let us say 
$J_{ac}J^{bc}=\delta_a{}^b$, where $\delta_a{}^b$ is the Kronecker delta). 
Thus, we may rewrite the coeffective complex as 
\begin{equation}\label{coeffective}
\Lambda_\perp^n\xrightarrow{\,d_\perp\,}\Lambda_\perp^{n-1}
\xrightarrow{\,d_\perp\,}\cdots\xrightarrow{\,d_\perp\,}\Lambda_\perp^2
\xrightarrow{\,d_\perp\,}\Lambda^1\xrightarrow{\,d_\perp\,}\Lambda^0\to 0.
\end{equation}
Bouche~\cite{Bo} showed that it is elliptic except at~$\Lambda_\perp^n$. Since 
the diagrams with exact rows
$$\begin{array}{ccccccccc}
0&\to&\Lambda^{k-2}&\xrightarrow{\,J\wedge\underbar\enskip\,}
&\Lambda^k&\to&\Lambda_\perp^k&\to&0\\
&&\downarrow&&\downarrow\\
0&\to&\Lambda^{k-1}&\xrightarrow{\,J\wedge\underbar\enskip\,}
&\Lambda^{k+1}&\to&\Lambda_\perp^{k+1}&\to&0
\end{array}$$
commute, there is a canonically defined differential complex going the other
way:
\begin{equation}\label{effective}
0\to\Lambda^0\xrightarrow{\,d\,}\Lambda^1\xrightarrow{\,d_\perp\,}
\Lambda_\perp^2\xrightarrow{\,d_\perp\,}\cdots\xrightarrow{\,d_\perp\,}
\Lambda_\perp^{n-1}\xrightarrow{\,d_\perp\,}\Lambda_\perp^n.\end{equation}
In fact, one can easily check that (\ref{coeffective}) and (\ref{effective}) 
are adjoint to each other under the pairing
$$\Lambda_\perp^k\otimes\Lambda_\perp^k
\stackrel{\simeq\quad}{\longrightarrow}
\Lambda_\perp^{2n-k}\otimes\Lambda_\perp^k
\xrightarrow{\,\underbar\enskip\,\wedge\underbar\enskip}\Lambda^{2n}.$$
The {\em Rumin-Seshadri\/} complex joins (\ref{coeffective}) and
(\ref{effective}) with a symplectically invariant {\em second order\/} linear
differential operator $d_\perp^{(2)}:\Lambda_\perp^n\to\Lambda_\perp^n$ to 
obtain an elliptic complex 
\begin{equation}
\label{RScomplex}\addtolength{\arraycolsep}{-1pt}\begin{array}{rcccccccccccc}
0&\to&\Lambda^0&\stackrel{d}{\longrightarrow}&\Lambda^1
&\stackrel{d_\perp}{\longrightarrow}&\Lambda_\perp^2
&\stackrel{d_\perp}{\longrightarrow}&\Lambda_\perp^3
&\stackrel{d_\perp}{\longrightarrow}&\cdots
&\stackrel{d_\perp}{\longrightarrow}&\Lambda_\perp^{n}\\[2pt]
&&&&&&&&&&&&\big\downarrow\makebox[0pt][l]{\scriptsize$d_\perp^{(2)}$}\\
.\,0&\leftarrow&\Lambda^0&\stackrel{d_\perp}{\longleftarrow}&\Lambda^1
&\stackrel{d_\perp}{\longleftarrow}&\Lambda_\perp^2
&\stackrel{d_\perp}{\longleftarrow}&\Lambda_\perp^3
&\stackrel{d_\perp}{\longleftarrow}&\cdots
&\stackrel{d_\perp}{\longleftarrow}&\Lambda_\perp^{n}
\end{array}\end{equation}
In four dimensions this complex is due to Smith~\cite{Sm} and in higher
dimensions it was also found by L.-S.~Tseng and S.-T.~Yau~\cite{TY} who go on
to study its cohomology on compact manifolds. The construction of
(\ref{RScomplex}) given in \cite{BEGN} will be generalised in the following
section.

\section{Conformally symplectic manifolds}
A {\em conformally symplectic structure\/} on an even dimensional manifold $M$
of dimension at least $6$ is defined by a non-degenerate $2$-form $J$ but,
instead of requiring that $J$ be closed, as one would for a symplectic
structure, one requires only that
\begin{equation}\label{def_alpha}dJ=2\alpha\wedge J\end{equation}
for some $1$-form~$\alpha$ (the factor of $2$ being chosen only for 
convenience). Non-degeneracy of $J$ implies that $\alpha$ is 
uniquely defined by~(\ref{def_alpha}). It is called the {\em Lee 
form\/}~\cite{L}. Differentiating (\ref{def_alpha}) gives
$$0=d^2J=2d\alpha\wedge J+2\alpha\wedge dJ
=2d\alpha\wedge J+4\alpha\wedge\alpha\wedge J=
2d\alpha\wedge J$$
and, as $J\wedge\underbar\enskip:\Lambda^2\to\Lambda^4$ is injective, we see
that $\alpha$ is closed. In dimension~$4$, equation (\ref{def_alpha}) defines a
unique Lee form $\alpha$ and, for the definition of conformally symplectic, we
require that $\alpha$ be closed. If we rescale $J$ by a positive smooth
function, say $\hat J=\Omega^2J$, then (\ref{def_alpha}) remains valid with 
$\alpha$ replaced by $\hat\alpha=\alpha+\Upsilon$ for 
$\Upsilon\equiv d\log\Omega$. Hence, the notion of conformally symplectic is 
invariant under such rescalings (and also in dimension $4$ since 
$d\Upsilon=0$). Locally, we may use this freedom to eliminate $\alpha$ and 
obtain an ordinary symplectic structure. Globally, however, this need not be 
the case. For example, the rescaled symplectic form 
$$J\equiv\left(1/{\|x\|}\right)^2(dx^1\wedge dx^2+dx^3\wedge dx^4+\cdots)$$
on ${\mathbb{R}}^{2n}$ is invariant under dilation $x\mapsto\lambda x$ and,
therefore, descends to a conformally symplectic structure on 
$S^1\times S^{2n-1}$ whereas there is no global symplectic form on this
manifold. If we continue to denote the inverse of $J_{ab}$ by $J^{ab}$, and
consider the vector field $X^a\equiv J^{ab}\alpha_b$, then the identities
$$\begin{array}{l}
J^{ad}J^{be}J^{cf}(\nabla_{[d}J_{ef]}\!-\!2\alpha_{[d}J_{ef]})=
J^{d[a}\nabla_dJ^{bc]}-2X^{[a}J^{bc]}\\[4pt]
J^{ad}J^{be}\big(2\nabla_{[d}\alpha_{e]}
+3X^c(\nabla_{[d}J_{ec]}\!-\!2\alpha_{[d}J_{ec]})\big)
=-X^c\nabla_{c}J^{ab}-2J^{c[a}\nabla_cX^{b]}\end{array}$$
are readily established for any torsion-free connection $\nabla_a$ and show
that a conformally symplectic structure is equivalent to a {\em Jacobi
structure\/} $(J^{ab},X^a)$ if we insist that $J^{ab}$ be non-degenerate (as
discussed in~\cite{Ba}). 
\begin{thm}\label{conf_symp_thm}
On any conformally symplectic manifold $(M,J)$, there is a canonically defined
elliptic complex
\begin{equation}\label{conformalRS}\begin{array}{rcccccccccccc}
0&\to&\Lambda^0&\to&\Lambda^1
&\to&\Lambda_\perp^2
&\to&\Lambda_\perp^3
&\to&\cdots
&\to&\Lambda_\perp^{n}\\[2pt]
&&&&&&&&&&&&\big\downarrow\\
0&\leftarrow&\Lambda^0&\leftarrow&\Lambda^1
&\leftarrow&\Lambda_\perp^2
&\leftarrow&\Lambda_\perp^3
&\leftarrow&\cdots
&\leftarrow&\Lambda_\perp^{n}
\end{array}\end{equation}
where $\Lambda_\perp^k$ denotes the bundle of $k$-forms that are trace-free
with respect to~$J$. All operators are first order except for the middle
operator, which is second order. In the symplectic case, the second half of
the complex coincides with the coeffective complex. This complex is locally 
exact except at $\Lambda^0$ and $\Lambda^1$ near the beginning.
\end{thm}
\begin{proof} Consider the diagram
\begin{equation}\label{double_complex}\begin{array}{ccccccl}
\longrightarrow&\Lambda^p&\xrightarrow{\,d-2\alpha\wedge\underbar\enskip\,}
&\Lambda^{p+1}&\xrightarrow{\,d-2\alpha\wedge\underbar\enskip\,}
&\Lambda^{p+2}&\longrightarrow\\
&\Big\uparrow\makebox[0pt][l]{\scriptsize$J\wedge\underbar\enskip$}
&&\Big\uparrow\makebox[0pt][l]{\scriptsize$J\wedge\underbar\enskip$}
&&\Big\uparrow\makebox[0pt][l]{\scriptsize$J\wedge\underbar\enskip$}\\
\longrightarrow&\Lambda^{p-2}&\xrightarrow{\quad d\quad}
&\Lambda^{p-1}&\xrightarrow{\quad d\quad}
&\Lambda^p&\longrightarrow.
\end{array}\end{equation}
The bottom row is the de~Rham complex and, in particular, is locally exact
except at~$\Lambda^0$. Since $d\alpha=0$, the same is true of the top row.
Since $dJ=2\alpha\wedge J$, the diagram commutes. Now consider the columns. In 
the middle, non-degeneracy of $J$ ensures that 
$$J\wedge\underbar\enskip:\Lambda^{n-1}\to\Lambda^{n+1}$$
is an isomorphism. To the left of this, we have injections and, to the right,
we have surjections. As discussed \S\ref{intro}, the trace-free forms 
$\Lambda_\perp^k$ may be canonically identified with the cokernel of 
$$J\wedge\underbar\enskip:\Lambda^{k-2}\to\Lambda^k
\enskip\mbox{for }k=2,3,\ldots,n$$
but also with the kernel of
$$J\wedge\underbar\enskip:\Lambda^{2n-k}\to\Lambda^{2n-k+2}
\enskip\mbox{for }k=2,3,\ldots,n.$$
The spectral sequence of a double complex completes the proof.
\end{proof}
\noindent Explicit formul{\ae} for the operators in this complex can be given
by using indices and an arbitrarily chosen torsion-free connection but are
quite complicated since they necessarily employ the decomposition of
arbitrary $k$-forms into their trace-free parts
$$\Lambda^k=\Lambda_\perp^k\oplus\Lambda_\perp^{k-2}\oplus
\Lambda_\perp^{k-4}\oplus\cdots \enskip\mbox{for }k=2,3,\ldots,n$$
corresponding to the branching of $\Lambda^k{\mathbb{R}}^{2n}$ under 
${\mathrm{Sp}}(2n,{\mathbb{R}})\subset{\mathrm{SL}}(2n,{\mathrm{R}})$ 
(cf.~the combinatorial formul{\ae} in~\cite[part II,\:\S2.1]{TY}).  

To discuss the global cohomology of the complex (\ref{conformalRS}) let us 
relabel its terms as $B^r$ for $r=0,1,2,\ldots,2n,2n+1$ and define
\begin{equation}\label{conformal_symplectic_cohomology}
H_J^r(M)\equiv\frac{\ker:\Gamma(M,B^r)\to\Gamma(M,B^{r+1})}
{\im:\Gamma(M,B^{r-1})\to\Gamma(M,B^r)}.\end{equation}
In comparison with~\cite{FIL} in the symplectic case, we have 
$$H_J^r(M)=H^{r-1}({\mathcal{A}}(M))\enskip\mbox{for }r=n+2,n+3,\ldots,2n+1$$ 
for their {\em coeffective cohomology\/} but now, for compact~$M$, we have
finite-dimensional vector spaces for all $r=0,1,2,\ldots,2n,2n+1$. Also in the 
symplectic case, these cohomologies were introduced and studied by Tseng and 
Yau~\cite{TY} and our notation compares as follows.
$$\begin{array}{ll}H_J^r(M)=PH_{\partial_+}^r(M)&\mbox{for }0\leq r<n\\[4pt]
H_J^n(M)=PH_{dd^\Lambda}^n(M)&H_J^{n+1}(M)=PH^n_{d+d^\Lambda}(M)\\[4pt]
H_J^r(M)=PH_{\partial_-}^{2n+1-r}(M)&\mbox{for }n+1<r\leq 2n+1 
\end{array}$$
(Tseng and Yau refer to these and similar cohomologies as `primitive.')
According to Theorem~\ref{conf_symp_thm}, the cohomology
${\mathcal{H}}^\bullet$ of (\ref{conformalRS}) on the level of sheaves of
germs of smooth functions occurs only at $B^0$ and $B^1$ and, from its proof,
we see that ${\mathcal{H}}^1={\mathbb{R}}$. Also ${\mathcal{H}}^0$ is a locally
constant sheaf. Specifically,
$${\mathcal{H}}^0=\{f\mbox{ s.t\ }df-2f\alpha=0\},$$ 
and may equivalently be viewed as parallel sections of the trivial bundle
equipped with the flat connection defined by $-2\alpha$ as connection form. In
the symplectic case, we have ${\mathcal{H}}^0={\mathbb{R}}$. Evidently, the top
row of (\ref{double_complex}) provides a fine resolution of ${\mathcal{H}}^0$
and so the sheaf cohomology $H^r(M,{\mathcal{H}}^0)$ may be identified as the
cohomology of the complex $\Gamma(M,\Lambda^\bullet)$ with 
$\omega\mapsto d\omega-2\alpha\wedge\omega$ as differential. The following 
theorem extends the long exact sequence~\cite[(5)]{FIL}.
\begin{thm}\label{cohomology}
On a conformally symplectic manifold $(M,J)$, we have
$$H_J^0(M)=H^0(M,{\mathcal{H}}^0),\quad 
H_J^{2n+1}(M)=H^{2n}(M,{\mathbb{R}}),$$
and a long exact sequence
$$\begin{array}l
0\to H^1(M,{\mathcal{H}}^0)\to H_J^1(M)\to H^0(M,{\mathbb{R}})
\xrightarrow{\,\delta\,}H^2(M,{\mathcal{H}}^0)\to\cdots\\
\phantom{0\to H^1(M,{\mathcal{H}}^0)}
\to H_J^r(M)\to H^{r-1}(M,{\mathbb{R}})
\xrightarrow{\,\delta\,}H^{r+1}(M,{\mathcal{H}}^0)\to\cdots\\
\phantom{0\to H^1(M,{\mathcal{H}}^0)}
\to H_J^{2n}(M)\to H^{2n-1}(M,{\mathbb{R}})
\to 0,\vphantom{\xrightarrow{\,\delta\,}}
\end{array}$$
where $\delta:H^{r-1}(M,{\mathbb{R}})\to H^{r+1}(M,{\mathcal{H}}^0)$ is given 
by cup product with the cohomology class $[J]\in H^2(M,{\mathcal{H}}^0)$.
\end{thm}
\begin{proof}
The hypercomology spectral sequence for the complex $B^\bullet$ as a complex 
of sheaves reads, at the $E_2$-level 
$$\begin{picture}(350,50)
\put(0,0){\vector(1,0){350}}
\put(0,0){\vector(0,1){50}}
\put(30,10){\makebox(0,0){$H^0(M,{\mathcal{H}}^0)$}}
\put(110,10){\makebox(0,0){$H^1(M,{\mathcal{H}}^0)$}}
\put(190,10){\makebox(0,0){$H^2(M,{\mathcal{H}}^0)$}}
\put(270,10){\makebox(0,0){$H^3(M,{\mathcal{H}}^0)$}}
\put(340,10){\makebox(0,0){$\cdots$}}
\put(30,40){\makebox(0,0){$H^0(M,{\mathbb{R}})$}}
\put(110,40){\makebox(0,0){$H^1(M,{\mathbb{R}})$}}
\put(190,40){\makebox(0,0){$H^2(M,{\mathbb{R}})$}}
\put(270,40){\makebox(0,0){$H^3(M,{\mathbb{R}})$}}
\put(340,40){\makebox(0,0){$\cdots$}}
\put(60,38){\vector(4,-1){100}}
\put(140,38){\vector(4,-1){100}}
\put(220,38){\vector(4,-1){100}}
\end{picture}$$
and the desired conclusions follow. 
\end{proof}
\noindent (Spectral sequence reasoning can always be replaced by an appropriate
diagram chase, in this case on the double complex~(\ref{double_complex}).)

As an application of Theorem~\ref{cohomology}, if we consider complex
projective space ${\mathbb{CP}}_n$ with $J$ its usual K\"ahler form, then
$${[J]}\cup\underbar\enskip:H^{r-1}({\mathbb{CP}}_n,{\mathbb{R}})\to
H^{r+1}({\mathbb{CP}}_n,{\mathbb{R}})$$
is an isomorphism for $1\leq r\leq 2n-1$. Therefore,
$$H^0_J({\mathbb{CP}})={\mathbb{R}},\quad 
H_J^r({\mathbb{CP}}_n)=0\enskip\mbox{for } 1\leq r\leq 2n,\quad
H_J^{2n+1}({\mathbb{CP}}_n)={\mathbb{R}}.$$
More generally, Theorem~\ref{cohomology} shows that the cohomology $H_J^r(M)$
of a symplectic manifold is determined by its de~Rham cohomology and the action
of the symplectic class $[J]\in H^2(M,{\mathbb{R}})$. In particular, there are
evident inequalities concerning $\dim H_J^r$ and the Betti numbers of a compact
symplectic manifold (including those of \cite[Theorem~3.1]{FIL}).

\section{Conformally calibrated $G_2$-manifolds}
Following~\cite{FU}, a {\em conformally calibrated $G_2$-manifold\/} is defined
as a $G_2$-manifold $(M,\phi)$ such that
\begin{equation}\label{G2alpha}d\phi=2\alpha\wedge\phi\end{equation}
for some $1$-form~$\alpha$. Recall~\cite{B,DNW,FU} that $\phi$ is the 
{\em fundamental $3$-form} defining a reduction of structure group on the
$7$-dimensional smooth manifold $M$ from ${\mathrm{GL}}(7,{\mathbb{R}})$ to
$G_2\subset{\mathrm{SO}}(7)\subset{\mathrm{GL}}(7,{\mathbb{R}})$. In parallel
with the symplectic case, the form $\phi$ may be locally rescaled so that it is
closed (and a $G_2$-manifold with closed fundamental form is said to be
calibrated.)
As in the symplectic case and as detailed in~\cite{FU},
the form~$\phi$, pointwise sometimes known as the {\em Cayley
form\/}~\cite{DNW}, is sufficiently non-degenerate that
\begin{equation}\label{linalg}\begin{array}{ll}
\phi\wedge\underbar\enskip:\Lambda^k\longrightarrow\Lambda^{k+3}
&\mbox{is injective for }k=0,1\\
\phi\wedge\underbar\enskip:\Lambda^2\stackrel{\simeq\quad}{\longrightarrow}
\Lambda^5\\
\phi\wedge\underbar\enskip:\Lambda^k\longrightarrow\Lambda^{k+3}
&\mbox{is surjective for }k=3,4.
\end{array}\end{equation}
One way to see this is to decompose the forms on $M$ into $G_2$-irreducibles
\begin{equation}\label{decompose}\begin{array}{c}
\Lambda^0=
\begin{picture}(24,5)
\put(5.4,0){\line(1,0){14.8}}
\put(4,1.6){\line(1,0){16}}
\put(5.4,3.2){\line(1,0){14.8}}
\put(4,1.2){\makebox(0,0){$\bullet$}}
\put(20,1.2){\makebox(0,0){$\bullet$}}
\put(12,1.5){\makebox(0,0){$\langle$}}
\put(4,8){\makebox(0,0){$\scriptstyle 0$}}
\put(20,8){\makebox(0,0){$\scriptstyle 0$}}
\end{picture}
\qquad\Lambda^1=
\begin{picture}(24,5)
\put(5.4,0){\line(1,0){14.8}}
\put(4,1.6){\line(1,0){16}}
\put(5.4,3.2){\line(1,0){14.8}}
\put(4,1.2){\makebox(0,0){$\bullet$}}
\put(20,1.2){\makebox(0,0){$\bullet$}}
\put(12,1.5){\makebox(0,0){$\langle$}}
\put(4,8){\makebox(0,0){$\scriptstyle 1$}}
\put(20,8){\makebox(0,0){$\scriptstyle 0$}}
\end{picture}
\qquad\Lambda^2=
\begin{picture}(24,5)
\put(5.4,0){\line(1,0){14.8}}
\put(4,1.6){\line(1,0){16}}
\put(5.4,3.2){\line(1,0){14.8}}
\put(4,1.2){\makebox(0,0){$\bullet$}}
\put(20,1.2){\makebox(0,0){$\bullet$}}
\put(12,1.5){\makebox(0,0){$\langle$}}
\put(4,8){\makebox(0,0){$\scriptstyle 0$}}
\put(20,8){\makebox(0,0){$\scriptstyle 1$}}
\end{picture}\oplus\begin{picture}(24,5)
\put(5.4,0){\line(1,0){14.8}}
\put(4,1.6){\line(1,0){16}}
\put(5.4,3.2){\line(1,0){14.8}}
\put(4,1.2){\makebox(0,0){$\bullet$}}
\put(20,1.2){\makebox(0,0){$\bullet$}}
\put(12,1.5){\makebox(0,0){$\langle$}}
\put(4,8){\makebox(0,0){$\scriptstyle 1$}}
\put(20,8){\makebox(0,0){$\scriptstyle 0$}}
\end{picture}\\[10pt]
\Lambda^3=
\begin{picture}(24,5)
\put(5.4,0){\line(1,0){14.8}}
\put(4,1.6){\line(1,0){16}}
\put(5.4,3.2){\line(1,0){14.8}}
\put(4,1.2){\makebox(0,0){$\bullet$}}
\put(20,1.2){\makebox(0,0){$\bullet$}}
\put(12,1.5){\makebox(0,0){$\langle$}}
\put(4,8){\makebox(0,0){$\scriptstyle 2$}}
\put(20,8){\makebox(0,0){$\scriptstyle 0$}}
\end{picture}\oplus\begin{picture}(24,5)
\put(5.4,0){\line(1,0){14.8}}
\put(4,1.6){\line(1,0){16}}
\put(5.4,3.2){\line(1,0){14.8}}
\put(4,1.2){\makebox(0,0){$\bullet$}}
\put(20,1.2){\makebox(0,0){$\bullet$}}
\put(12,1.5){\makebox(0,0){$\langle$}}
\put(4,8){\makebox(0,0){$\scriptstyle 1$}}
\put(20,8){\makebox(0,0){$\scriptstyle 0$}}
\end{picture}\oplus\begin{picture}(24,5)
\put(5.4,0){\line(1,0){14.8}}
\put(4,1.6){\line(1,0){16}}
\put(5.4,3.2){\line(1,0){14.8}}
\put(4,1.2){\makebox(0,0){$\bullet$}}
\put(20,1.2){\makebox(0,0){$\bullet$}}
\put(12,1.5){\makebox(0,0){$\langle$}}
\put(4,8){\makebox(0,0){$\scriptstyle 0$}}
\put(20,8){\makebox(0,0){$\scriptstyle 0$}}
\end{picture}\qquad\Lambda^4=
\begin{picture}(24,5)
\put(5.4,0){\line(1,0){14.8}}
\put(4,1.6){\line(1,0){16}}
\put(5.4,3.2){\line(1,0){14.8}}
\put(4,1.2){\makebox(0,0){$\bullet$}}
\put(20,1.2){\makebox(0,0){$\bullet$}}
\put(12,1.5){\makebox(0,0){$\langle$}}
\put(4,8){\makebox(0,0){$\scriptstyle 2$}}
\put(20,8){\makebox(0,0){$\scriptstyle 0$}}
\end{picture}\oplus\begin{picture}(24,5)
\put(5.4,0){\line(1,0){14.8}}
\put(4,1.6){\line(1,0){16}}
\put(5.4,3.2){\line(1,0){14.8}}
\put(4,1.2){\makebox(0,0){$\bullet$}}
\put(20,1.2){\makebox(0,0){$\bullet$}}
\put(12,1.5){\makebox(0,0){$\langle$}}
\put(4,8){\makebox(0,0){$\scriptstyle 1$}}
\put(20,8){\makebox(0,0){$\scriptstyle 0$}}
\end{picture}\oplus\begin{picture}(24,5)
\put(5.4,0){\line(1,0){14.8}}
\put(4,1.6){\line(1,0){16}}
\put(5.4,3.2){\line(1,0){14.8}}
\put(4,1.2){\makebox(0,0){$\bullet$}}
\put(20,1.2){\makebox(0,0){$\bullet$}}
\put(12,1.5){\makebox(0,0){$\langle$}}
\put(4,8){\makebox(0,0){$\scriptstyle 0$}}
\put(20,8){\makebox(0,0){$\scriptstyle 0$}}
\end{picture}\\[10pt]
\Lambda^5=
\begin{picture}(24,5)
\put(5.4,0){\line(1,0){14.8}}
\put(4,1.6){\line(1,0){16}}
\put(5.4,3.2){\line(1,0){14.8}}
\put(4,1.2){\makebox(0,0){$\bullet$}}
\put(20,1.2){\makebox(0,0){$\bullet$}}
\put(12,1.5){\makebox(0,0){$\langle$}}
\put(4,8){\makebox(0,0){$\scriptstyle 0$}}
\put(20,8){\makebox(0,0){$\scriptstyle 1$}}
\end{picture}\oplus\begin{picture}(24,5)
\put(5.4,0){\line(1,0){14.8}}
\put(4,1.6){\line(1,0){16}}
\put(5.4,3.2){\line(1,0){14.8}}
\put(4,1.2){\makebox(0,0){$\bullet$}}
\put(20,1.2){\makebox(0,0){$\bullet$}}
\put(12,1.5){\makebox(0,0){$\langle$}}
\put(4,8){\makebox(0,0){$\scriptstyle 1$}}
\put(20,8){\makebox(0,0){$\scriptstyle 0$}}
\end{picture}\qquad
\Lambda^6=
\begin{picture}(24,5)
\put(5.4,0){\line(1,0){14.8}}
\put(4,1.6){\line(1,0){16}}
\put(5.4,3.2){\line(1,0){14.8}}
\put(4,1.2){\makebox(0,0){$\bullet$}}
\put(20,1.2){\makebox(0,0){$\bullet$}}
\put(12,1.5){\makebox(0,0){$\langle$}}
\put(4,8){\makebox(0,0){$\scriptstyle 1$}}
\put(20,8){\makebox(0,0){$\scriptstyle 0$}}
\end{picture}\qquad
\Lambda^7=
\begin{picture}(24,5)
\put(5.4,0){\line(1,0){14.8}}
\put(4,1.6){\line(1,0){16}}
\put(5.4,3.2){\line(1,0){14.8}}
\put(4,1.2){\makebox(0,0){$\bullet$}}
\put(20,1.2){\makebox(0,0){$\bullet$}}
\put(12,1.5){\makebox(0,0){$\langle$}}
\put(4,8){\makebox(0,0){$\scriptstyle 0$}}
\put(20,8){\makebox(0,0){$\scriptstyle 0$}}
\end{picture}
\end{array}\end{equation}
and check (\ref{linalg}) on the level of highest weights. The canonical Hodge
isomorphism $\Lambda^k\cong\Lambda^{7-k}$ is evident in this decomposition.
Parallel to the symplectic case let us write
$$\Lambda_\perp^4\equiv\ker\phi\wedge\underbar\enskip:\Lambda^3\to\Lambda^6
\qquad\qquad
\Lambda_\perp^3\equiv\ker\phi\wedge\underbar\enskip:\Lambda^4\to\Lambda^7$$
and, by inspecting~(\ref{decompose}), note that 
$$\Lambda_\perp^3=\begin{picture}(24,5)
\put(5.4,0){\line(1,0){14.8}}
\put(4,1.6){\line(1,0){16}}
\put(5.4,3.2){\line(1,0){14.8}}
\put(4,1.2){\makebox(0,0){$\bullet$}}
\put(20,1.2){\makebox(0,0){$\bullet$}}
\put(12,1.5){\makebox(0,0){$\langle$}}
\put(4,8){\makebox(0,0){$\scriptstyle 2$}}
\put(20,8){\makebox(0,0){$\scriptstyle 0$}}
\end{picture}\oplus\begin{picture}(24,5)
\put(5.4,0){\line(1,0){14.8}}
\put(4,1.6){\line(1,0){16}}
\put(5.4,3.2){\line(1,0){14.8}}
\put(4,1.2){\makebox(0,0){$\bullet$}}
\put(20,1.2){\makebox(0,0){$\bullet$}}
\put(12,1.5){\makebox(0,0){$\langle$}}
\put(4,8){\makebox(0,0){$\scriptstyle 1$}}
\put(20,8){\makebox(0,0){$\scriptstyle 0$}}
\end{picture}\qquad\qquad
\Lambda_\perp^4=\begin{picture}(24,5)
\put(5.4,0){\line(1,0){14.8}}
\put(4,1.6){\line(1,0){16}}
\put(5.4,3.2){\line(1,0){14.8}}
\put(4,1.2){\makebox(0,0){$\bullet$}}
\put(20,1.2){\makebox(0,0){$\bullet$}}
\put(12,1.5){\makebox(0,0){$\langle$}}
\put(4,8){\makebox(0,0){$\scriptstyle 2$}}
\put(20,8){\makebox(0,0){$\scriptstyle 0$}}
\end{picture}\oplus\begin{picture}(24,5)
\put(5.4,0){\line(1,0){14.8}}
\put(4,1.6){\line(1,0){16}}
\put(5.4,3.2){\line(1,0){14.8}}
\put(4,1.2){\makebox(0,0){$\bullet$}}
\put(20,1.2){\makebox(0,0){$\bullet$}}
\put(12,1.5){\makebox(0,0){$\langle$}}
\put(4,8){\makebox(0,0){$\scriptstyle 0$}}
\put(20,8){\makebox(0,0){$\scriptstyle 0$}}
\end{picture}$$
also provide canonical complements to the ranges of 
$\phi\wedge\underbar\enskip:\Lambda^0\to\Lambda^3$ and
$\phi\wedge\underbar\enskip:\Lambda^1\to\Lambda^4$, respectively. {From}
(\ref{linalg}) we see that, as in the conformally symplectic case, $\alpha$ is
uniquely defined by (\ref{G2alpha}) and is closed.

\begin{thm}\label{conf_cal_G2}
On any conformally calibrated $G_2$-manifold $(M,\phi)$, there is a 
canonically defined elliptic complex
\begin{equation}\label{confG2complex}\begin{array}{rcccccccccc}
0&\to&\Lambda^0&\to&\Lambda^1&\to&\Lambda^2&\to&\Lambda_\perp^3
&\to&\Lambda_\perp^4\\[2pt]
&&&&&&&&&&\big\downarrow\\[2pt]
0&\leftarrow&\Lambda^0
&\leftarrow&\Lambda^1&\leftarrow&\Lambda^2&\leftarrow&\Lambda_\perp^3
&\leftarrow&\Lambda_\perp^4
\end{array}\end{equation}
All differential operators are first order except for the middle operator,
which is second order. The second half of this complex coincides with the
coeffective complex defined in\/~\cite{FIL}. It is locally exact except
at $\Lambda^0$ and $\Lambda^2$ near the beginning.
\end{thm}
\begin{proof} Consider the diagram
$$\begin{array}{ccccccl}
\longrightarrow&\Lambda^p&\xrightarrow{\,d-2\alpha\wedge\underbar\enskip\,}
&\Lambda^{p+1}&\xrightarrow{\,d-2\alpha\wedge\underbar\enskip\,}
&\Lambda^{p+2}&\longrightarrow\\
&\Big\uparrow\makebox[0pt][l]{\scriptsize$\underbar\enskip\wedge\phi$}
&&\Big\uparrow\makebox[0pt][l]{\scriptsize$\underbar\enskip\wedge\phi$}
&&\Big\uparrow\makebox[0pt][l]{\scriptsize$\underbar\enskip\wedge\phi$}\\
\longrightarrow&\Lambda^{p-3}&\xrightarrow{\quad d\quad}
&\Lambda^{p-2}&\xrightarrow{\quad d\quad}
&\Lambda^{p-1}&\longrightarrow.
\end{array}$$
The bottom row is the de~Rham complex and, in particular, is locally exact
except at~$\Lambda^0$. Since $d\alpha=0$, the same is true of the top row.
Since $d\phi=2\alpha\wedge\phi$, the diagram commutes. The columns behave
according to~(\ref{linalg}). Hence, the first spectral sequence of this 
double complex reads, at the $E_1$-level
$$\begin{picture}(310,40)
\put(0,0){\vector(1,0){310}}
\put(0,0){\vector(0,1){40}}
\put(10,10){\makebox(0,0){$0$}}
\put(45,10){\makebox(0,0){$0$}}
\put(80,10){\makebox(0,0){$0$}}
\put(115,10){\makebox(0,0){$0$}}
\put(150,10){\makebox(0,0){$0$}}
\put(185,9){\makebox(0,0){$\Lambda_\perp^3$}}
\put(220,10){\makebox(0,0){$\Lambda^4$}}
\put(255,10){\makebox(0,0){$\Lambda^5$}}
\put(290,10){\makebox(0,0){$\Lambda^7$}}
\put(10,30){\makebox(0,0){$\Lambda^0$}}
\put(45,30){\makebox(0,0){$\Lambda^1$}}
\put(80,30){\makebox(0,0){$\Lambda^2$}}
\put(115,29){\makebox(0,0){$\Lambda_\perp^3$}}
\put(150,30){\makebox(0,0){$0$}}
\put(185,30){\makebox(0,0){$0$}}
\put(220,30){\makebox(0,0){$0$}}
\put(255,30){\makebox(0,0){$0$}}
\put(290,30){\makebox(0,0){$0$}}
\put(202,8){\makebox(0,0){$\to$}}
\put(237,8){\makebox(0,0){$\to$}}
\put(272,8){\makebox(0,0){$\to$}}
\put(27,28){\makebox(0,0){$\to$}}
\put(62,28){\makebox(0,0){$\to$}}
\put(97,28){\makebox(0,0){$\to$}}
\end{picture}$$
Passing to the $E_2$-level constructs the complex and the second spectral 
sequence identifies its local cohomology ${\mathcal{H}}^\bullet$ as 
$${\mathcal{H}}^0=\{f\mbox{ s.t.\ }df-2f\alpha=0\},\quad
{\mathcal{H}}^2={\mathbb{R}},$$
with all others vanishing. Finally, ellipticity of this complex is inherited
from that of the de~Rham complex. Specifically, for $\Lambda^1\ni\xi\not=0$, 
the symbol complex of (\ref{confG2complex}) is constructed from the double 
complex 
$$\begin{array}{ccccccl}
\longrightarrow&\Lambda^p&\xrightarrow{\,\xi\wedge\underbar\enskip\,}
&\Lambda^{p+1}&\xrightarrow{\xi\wedge\underbar\enskip\,}
&\Lambda^{p+2}&\longrightarrow\\
&\Big\uparrow\makebox[0pt][l]{\scriptsize$\underbar\enskip\wedge\phi$}
&&\Big\uparrow\makebox[0pt][l]{\scriptsize$\underbar\enskip\wedge\phi$}
&&\Big\uparrow\makebox[0pt][l]{\scriptsize$\underbar\enskip\wedge\phi$}\\
\longrightarrow&\Lambda^{p-3}&\xrightarrow{\,\xi\wedge\underbar\enskip\,}
&\Lambda^{p-2}&\xrightarrow{\,\xi\wedge\underbar\enskip\,}
&\Lambda^{p-1}&\longrightarrow,
\end{array}$$
the rows of which are exact (they are Koszul complexes).
\end{proof}
As in the (conformally) symplectic case, this construction (and this proof of
ellipticity) avoids explicit formul{\ae} for the operators. If such formul{\ae}
are needed, then one simply needs explicitly to write out the 
branching~(\ref{linalg}) (as is done in~\cite[p.~365]{FU}). 

As in the conformally symplectic case~(\ref{conformal_symplectic_cohomology}),
we may consider the global cohomology on $M$ of the
complex~(\ref{confG2complex}), which we shall denote by $H_\phi^r(M)$ for
$0\leq r\leq 9$.
\begin{thm}
On a conformally calibrated $G_2$-manifold $(M,\phi)$, we have
$$\begin{array}{ll}H_\phi^0(M)=H^0(M,{\mathcal{H}}^0)\quad
&H_\phi^1(M)=H^1(M,{\mathcal{H}}^0)\\[4pt]
H_\phi^8(M)=H^6(M,{\mathbb{R}})
&H_\phi^9(M)=H^7(M,{\mathbb{R}})
\end{array}$$
and a long exact sequence
$$\begin{array}l
0\to H^2(M,{\mathcal{H}}^0)\to H_\phi^2(M)\to H^0(M,{\mathbb{R}})
\xrightarrow{\,\delta\,}H^3(M,{\mathcal{H}}^0)\\
\phantom{0\to H^1(M,{\mathcal{H}}^0)}
\to H_\phi^3(M)\to H^1(M,{\mathbb{R}})
\xrightarrow{\,\delta\,}H^4(M,{\mathcal{H}}^0)\to\cdots\\
\phantom{0\to H^1(M,{\mathcal{H}}^0)}
\makebox[0pt][r]{$\cdots$}\to H_\phi^6(M)\to H^4(M,{\mathbb{R}})
\xrightarrow{\,\delta\,}H^7(M,{\mathcal{H}}^0)\\
\phantom{0\to H^1(M,{\mathcal{H}}^0)}
\to H_\phi^7(M)\to H^5(M,{\mathbb{R}})
\to 0,\vphantom{\xrightarrow{\,\delta\,}}
\end{array}$$
where $\delta:H^r(M,{\mathbb{R}})\to H^{r+3}(M,{\mathcal{H}}^0)$ is given 
by cup product with the cohomology class $[\phi]\in H^3(M,{\mathcal{H}}^0)$.
\end{thm}
\begin{proof}Immediate from the hypercohomology spectral sequence as for the
proof of Theorem~\ref{cohomology} except that the connecting homomorphism
$\delta$ does not appear until the $E_3$-level.
\end{proof}
In the calibrated case (when $\alpha=0$),
$H^r(M,{\mathcal{H}}^0)=H^r(M,{\mathbb{R}})$ and we see that $H_\phi^r(M)$ is
determined by the de~Rham cohomology of $M$ and the action of $[\phi]\in
H^3(M,{\mathbb{R}})$ by cup product.

\section{Other geometries}
There are several other geometries defined by special $k$-forms for which one
can apply similar reasoning. Certainly, there are
${\mathrm{Spin}}(7)$-geometries in dimension $8$ defined~\cite{B} by a fundamental
$4$-form~$\Phi$. The construction given in this article extends to this 
case and, by the work of Joyce~\cite{J}, there are non-trivial compact 
examples with $d\Phi=0$. 

Also, there are ${\mathrm{SO}}(3)\times{\mathrm{SO}}(3)$-geometries in
dimension $9$ defined~\cite{FN} by a fundamental $5$-form and
${\mathrm{SU}}(4)\times{\mathrm{U}}(1)$-geometries in dimension $10$
defined~\cite{DNW} by a fundamental $6$-form or $4$-form but, for the moment,
it is unclear whether there are any examples of such geometries that are not
locally homogeneous.

\end{document}